\documentclass[11pt,leqno]{article}
\usepackage[margin=1in]{geometry} 
\usepackage{amssymb,amsfonts,amsmath,bbm,mathrsfs,stmaryrd,mathtools}
\usepackage{xcolor}
\usepackage{url}

\usepackage{graphicx}

\usepackage{accents}

\usepackage{extarrows}

\usepackage[shortlabels]{enumitem}
\usepackage{tensor}

\usepackage{xr}
\usepackage[T1]{fontenc}
\usepackage[utf8]{inputenc}

\usepackage[colorlinks,
linkcolor=black!75!red,
citecolor=blue,
pdftitle={},
pdfproducer={pdfLaTeX},
pdfpagemode=None,
bookmarksopen=true,
bookmarksnumbered=true,
backref=page]{hyperref}

\usepackage{tikz}
\usetikzlibrary{arrows,calc,decorations.pathreplacing,decorations.markings,decorations.shapes,intersections,shapes.geometric,through,fit,shapes.symbols,positioning,decorations.pathmorphing}

\makeatletter
\newlength\zig@L
\newlength\zig@La
\newlength\zig@Lb

\newcommand{\xzigrightarrow}[2][]{%
  \mathrel{%
    \settowidth{\zig@La}{$\scriptstyle #2$}%
    \settowidth{\zig@Lb}{$\scriptstyle #1$}%
    \zig@L=\zig@La\relax
    \ifdim\zig@Lb>\zig@L \zig@L=\zig@Lb\fi
    \advance\zig@L by 2.2em\relax
    \tikz[baseline=-0.65ex]{%
      \draw[->,
            line cap=round,
            decorate,
            decoration={zigzag,segment length=4pt,amplitude=1.1pt}]%
        (0,0) -- (\zig@L,0)
        node[midway,above=2pt] {$\scriptstyle #2$}%
        \if\relax\detokenize{#1}\relax\else
          node[midway,below=2pt] {$\scriptstyle #1$}%
        \fi
      ;
    }%
  }%
}
\makeatother

\makeatletter
\newcommand{\squigjoin}{1mu} 

\def\sqleft@{\sim}                    
\def\sqmid@{\sim\mkern-\squigjoin}    

\def\rightsquigarrowfill@{%
  \arrowfill@{\sqleft@}{\sqmid@}{\mkern-4mu\succ}%
}

\newcommand{\xrightsquigarrow}[2][]{%
  \ext@arrow 0359\rightsquigarrowfill@{#1}{#2}%
}
\makeatother

\makeatletter
\newcommand*\circled[1]{\tikz[baseline=(char.base)]{
    \node[shape=circle, draw, inner sep=0pt, 
    minimum height={\f@size},] (char) {\vphantom{WAH1g}#1};}}
\makeatother

\makeatletter
\DeclareRobustCommand\widecheck[1]{{\mathpalette\@widecheck{#1}}}
\def\@widecheck#1#2{%
    \setbox\z@\hbox{\m@th$#1#2$}%
    \setbox\tw@\hbox{\m@th$#1%
       \widehat{%
          \vrule\@width\z@\@height\ht\z@
          \vrule\@height\z@\@width\wd\z@}$}%
    \dp\tw@-\ht\z@
    \@tempdima\ht\z@ \advance\@tempdima2\ht\tw@ \divide\@tempdima\thr@@
    \setbox\tw@\hbox{%
       \raise\@tempdima\hbox{\scalebox{1}[-1]{\lower\@tempdima\box
\tw@}}}%
    {\ooalign{\box\tw@ \cr \box\z@}}}
\makeatother

\usepackage{braket}

\usepackage[amsmath,thmmarks,hyperref]{ntheorem}
\usepackage{cleveref}

\newcommand\nthalias[1]{\AddToHook{env/#1/begin}{\crefalias{lemma}{#1}}}

\nthalias{assumption}
\nthalias{conjecture}
\nthalias{construction}
\nthalias{convention}
\nthalias{corollaryN}
\nthalias{corollary}
\nthalias{definition}
\nthalias{examples}
\nthalias{example}
\nthalias{lemma}
\nthalias{notation}
\nthalias{propositionN}
\nthalias{proposition}
\nthalias{question}
\nthalias{recollection}
\nthalias{remarks}
\nthalias{remark}
\nthalias{sketch}
\nthalias{theoremN}
\nthalias{theorem}

\creflabelformat{enumi}{#2#1#3}

\crefname{section}{Section}{Sections}
\crefformat{section}{#2Section~#1#3} 
\Crefformat{section}{#2Section~#1#3} 

\crefname{subsection}{\S}{\S\S}
\AtBeginDocument{%
  \crefformat{subsection}{#2\S#1#3}%
  \Crefformat{subsection}{#2\S#1#3}%
}

\crefname{subsubsection}{\S}{\S\S}
\AtBeginDocument{%
  \crefformat{subsubsection}{#2\S#1#3}%
  \Crefformat{subsubsection}{#2\S#1#3}%
}

\theoremstyle{plain}

\newtheorem{lemma}{Lemma}[section]
\newtheorem{proposition}[lemma]{Proposition}
\newtheorem{corollary}[lemma]{Corollary}
\newtheorem{theorem}[lemma]{Theorem}

\theoremstyle{plain}
\theoremnumbering{Alph}

\theoremstyle{plain}
\theorembodyfont{\upshape}
\theoremsymbol{\ensuremath{\blacklozenge}}

\newtheorem{definition}[lemma]{Definition}
\newtheorem{example}[lemma]{Example}

\newtheorem{remark}[lemma]{Remark}
\newtheorem{remarks}[lemma]{Remarks}

\newtheorem{recollection}[lemma]{Recollection}

\crefname{definition}{definition}{definitions}
\crefformat{definition}{#2definition~#1#3} 
\Crefformat{definition}{#2Definition~#1#3} 

\crefname{ex}{example}{examples}
\crefformat{example}{#2example~#1#3} 
\Crefformat{example}{#2Example~#1#3} 

\crefname{exs}{example}{examples}
\crefformat{examples}{#2example~#1#3} 
\Crefformat{examples}{#2Example~#1#3} 

\crefname{remark}{remark}{remarks}
\crefformat{remark}{#2remark~#1#3} 
\Crefformat{remark}{#2Remark~#1#3} 

\crefname{remarks}{remark}{remarks}
\crefformat{remarks}{#2remark~#1#3} 
\Crefformat{remarks}{#2Remark~#1#3} 

\crefname{convention}{convention}{conventions}
\crefformat{convention}{#2convention~#1#3} 
\Crefformat{convention}{#2Convention~#1#3} 

\crefname{notation}{notation}{notations}
\crefformat{notation}{#2notation~#1#3} 
\Crefformat{notation}{#2Notation~#1#3} 

\crefname{table}{table}{tables}
\crefformat{table}{#2table~#1#3} 
\Crefformat{table}{#2Table~#1#3}

\crefname{lemma}{lemma}{lemmas}
\crefformat{lemma}{#2lemma~#1#3} 
\Crefformat{lemma}{#2Lemma~#1#3} 

\crefname{proposition}{proposition}{propositions}
\crefformat{proposition}{#2proposition~#1#3} 
\Crefformat{proposition}{#2Proposition~#1#3} 

\crefname{propositionN}{proposition}{propositions}
\crefformat{propositionN}{#2proposition~#1#3} 
\Crefformat{propositionN}{#2Proposition~#1#3} 

\crefname{corollary}{corollary}{corollaries}
\crefformat{corollary}{#2corollary~#1#3} 
\Crefformat{corollary}{#2Corollary~#1#3} 

\crefname{corollaryN}{corollary}{corollaries}
\crefformat{corollaryN}{#2corollary~#1#3} 
\Crefformat{corollaryN}{#2Corollary~#1#3} 

\crefname{theorem}{theorem}{theorems}
\crefformat{theorem}{#2theorem~#1#3} 
\Crefformat{theorem}{#2Theorem~#1#3} 

\crefname{theoremN}{theorem}{theorems}
\crefformat{theoremN}{#2theorem~#1#3} 
\Crefformat{theoremN}{#2Theorem~#1#3} 

\crefname{enumi}{}{}
\crefformat{enumi}{#2#1#3}
\Crefformat{enumi}{#2#1#3}

\crefname{assumption}{assumption}{Assumptions}
\crefformat{assumption}{#2assumption~#1#3} 
\Crefformat{assumption}{#2Assumption~#1#3} 

\crefname{construction}{construction}{Constructions}
\crefformat{construction}{#2construction~#1#3} 
\Crefformat{construction}{#2Construction~#1#3} 

\crefname{sketch}{sketch}{Sketches}
\crefformat{sketch}{#2sketch~#1#3} 
\Crefformat{sketch}{#2Sketch~#1#3} 

\crefname{recollection}{recollection}{Recollections}
\crefformat{recollection}{#2recollection~#1#3} 
\Crefformat{recollection}{#2Recollection~#1#3} 

\crefname{question}{question}{Questions}
\crefformat{question}{#2question~#1#3} 
\Crefformat{question}{#2Question~#1#3} 

\crefname{equation}{}{}
\crefformat{equation}{(#2#1#3)} 
\Crefformat{equation}{(#2#1#3)}

\numberwithin{equation}{section}

\theoremstyle{nonumberplain}
\theoremsymbol{\ensuremath{\blacksquare}}

\newtheorem{proof}{Proof}
\newcommand\pf[1]{\newtheorem{#1}{Proof of \Cref{#1}}}

\newcommand\bG{{\mathbb G}}
\newcommand\bH{{\mathbb H}}

\newcommand\bK{{\mathbb K}}

\newcommand\bS{{\mathbb S}}
\newcommand\bT{{\mathbb T}}
\newcommand\bU{{\mathbb U}}

\newcommand\bZ{{\mathbb Z}}

\newcommand\cA{{\mathcal A}}
\newcommand\cB{{\mathcal B}}

\newcommand\cE{{\mathcal E}}
\newcommand\cF{{\mathcal F}}
\newcommand\cG{{\mathcal G}}

\newcommand\cM{{\mathcal M}}

\newcommand\cO{{\mathcal O}}
\newcommand\cP{{\mathcal P}}

\newcommand\cU{{\mathcal U}}

\newcommand\cW{{\mathcal W}}
\newcommand\cX{{\mathcal X}}

\newcommand\ol{\overline}
\newcommand\wt{\widetilde}

\DeclareMathOperator{\Ad}{Ad}

\newcommand{\Rep}[0]{\cat{Rep}}

\newcommand{\cat}[1]{\textsc{#1}}

\newcommand\spr[1]{\cite[\href{https://stacks.math.columbia.edu/tag/#1}{Tag {#1}}]{stacks-project}}
\newcommand{\qedhere}{\mbox{}\hfill\ensuremath{\blacksquare}}

\newcommand{\comment}[1]{}

\newcommand{\xrightarrowdbl}[2][]{%
  \xrightarrow[#1]{#2}\mathrel{\mkern-14mu}\rightarrow
}

\title{Equivariant compactifications, trivial embeddability and finite type}
\author{Alexandru Chirvasitu}

\begin{document}

\date{}

\newcommand{\Addresses}{{
  \bigskip
  \footnotesize

  \textsc{Department of Mathematics, University at Buffalo}
  \par\nopagebreak
  \textsc{Buffalo, NY 14260-2900, USA}  
  \par\nopagebreak
  \textit{E-mail address}: \texttt{achirvas@buffalo.edu}

}}

\maketitle

\begin{abstract}
  We characterize finite-type $\mathbb{G}$-principal $\mathbb{U}$-equivariant bundles on normal $\mathbb{U}$-spaces for compact Lie groups $\mathbb{U}$ and $\mathbb{G}$, in several ways, including (a) their extensibility across the $\mathbb{U}$-equivariant compactification $\beta_{\mathbb{U}}X$ and (b) their becoming finite-type upon extending the structure group along at least one $\mathbb{U}$-equivariant compact-Lie-group embedding $\mathbb{G}\le \mathbb{K}$. This generalizes non-equivariant results of Phillips and the author's characterizing finite-type matrix-algebra bundles, upon specializing $\mathbb{G}$ to projective unitary groups. 

  When the $\mathbb{U}$-action on $X$ has virtually abelian isotropy, matrix-algebra equivariant bundles are also finite-type precisely when, locally over a finite open $\mathbb{U}$-cover, they are tensor factors of trivial matrix bundles. In a $K$-theoretic offshoot we prove that for $\mathbb{U}$-actions with finite isotropy groups on compact Hausdorff spaces $X$ equivariant vector bundles $\mathcal{E}\to X$ are factors of trivial bundles $K$-theoretically: there is a class $a\in K_{\mathbb{U}}(X)$ with $[\mathcal{E}]a$ the class of a bundle induced by a $\mathbb{U}$-representation (which furthermore can be chosen so as to restrict to isotropy groups to multiples of the regular representations). This generalizes a result of Donovan and Karoubi.  
\end{abstract}


\vspace{.5cm}

\noindent{MSC 2020: 55R91; 55R10; 55R25; 22D10; 19L47; 46M20; 54D35; 55P91
  
}

\tableofcontents

\section*{Introduction}

The paper is concerned with \emph{$\bU$-equivariant bundles} (or \emph{bundles$_{\bU}$}, following the language of \cite[Definition 0.1]{MR5060733}) in the sense of \cite[\SS I.8, I.9]{td_transf-gp}: bundles $\cE\xrightarrow{} X$ over a base space, either $\bG$-principal for a topological (virtually exclusively compact Lie) group $\bG$ or vector/algebra/$C^*$-algebra/etc., with all maps and structures involved equivariant with respect to $\bU$-actions by the topological (again always compact Lie) group $\bU$. The notion is thus broader than, say, \cite[\S 1]{MR650393} (where the action $\bU\circlearrowright \bG$ is trivial) and narrower than \cite[pp.265-266]{MR885537} (or \cite[\S VII.1]{may_eq-hotop_1996}). 

Recall \cite[Definition I.9.1]{td_transf-gp} that \emph{finite-type} bundles are those trivializable over finite open covers (by invariant open subsets if working equivariantly). \Cref{th:ft.iff.ext.cpct} is in large part motivated by \cite[Proposition 2.9]{zbMATH05172034}, giving a number of equivalent characterizations for finite-type matrix-algebra bundles over normal spaces\footnote{The statement only assumes $T_{3\frac 12}$ separation, but see the caveats in \cite[Remark 1.11(1)]{2605.10752v1}.}. Among these are
\begin{itemize}[wide]
\item having finite type as a plain vector bundle;  

\item and extending across some \emph{compactification} of the base space (equivalently, the largest such: $\beta X$, the \emph{Stone-\v{C}ech compactification} of \cite[\S 6.2]{gj_rings}, say). 
\end{itemize}
\Cref{th:ft.iff.ext.cpct} extends this to the broader context naturally housing those characterizations: 
\begin{itemize}[wide]
\item by superimposing equivariance with respect to a compact Lie group $\bU$ operating on all objects/maps concerned;
\item and substituting for matrix/vector bundles principal bundles with compact Lie structure group $\bG$ acted upon by $\bU$. 
\end{itemize}

The statement is as follows:

\begin{theorem}\label{th:ft.iff.ext.cpct}
  Let $\bU$ be a compact Lie group acting on another $\bG$, and $X$ a $T_4$ $\bU$-space. The following conditions on a principal $\bG$-bundle$_{\bU}$ $\cP\xrightarrowdbl{}X$ are equivalent.
  \begin{enumerate}[(a),wide]
  \item\label{item:th:ft.iff.ext.cpct:ft} $\cP$ is of finite type.

  \item\label{item:th:ft.iff.ext.cpct:fin.cw} $\cP$ pulls back from a principal $\bG$-bundle$_{\bU}$ on a finite \emph{$\bU$-CW complex} \cite[\S I.3]{may_eq-hotop_1996}. 
    
  \item\label{item:th:ft.iff.ext.cpct:univ.cpct} $\cP$ pulls back from a principal $\bG$-bundle$_{\bU}$ on the universal equivariant compactification $X\lhook\joinrel\to \beta_{\bU}X$.
    
  \item\label{item:th:ft.iff.ext.cpct:cpct} $\cP$ pulls back through some equivariant compactification $X\to \ol{X}$.

  \item\label{item:th:ft.iff.ext.cpct:all.ind} The induced \cite[post Theorem I.8.15]{td_transf-gp} principal $\bK$-bundle$_{\bU}$ $\cP\times_{\bG}\bK$ is of finite type for every $\bU$-equivariant compact-Lie-group morphism $\bG\to \bK$.

  \item\label{item:th:ft.iff.ext.cpct:all.rep} $\cP\times_{\bG}V$ is of finite type for every $\bU$-equivariant (real or complex) $\bG$-representation $\bG\circlearrowright V$. 

  \item\label{item:th:ft.iff.ext.cpct:some.rep} $\cP\times_{\bG}V$ is of finite type for some faithful $\bU$-equivariant $\bG$-representation $\bG\circlearrowright V$. 
    
  \item\label{item:th:ft.iff.ext.cpct:some.ind} $\cP\times_{\bG}\bK$ is of finite type for some $\bU$-equivariant compact-Lie-group embedding $\bG\to \bK$.
  \end{enumerate}
\end{theorem}

This does all specialize back to the relevant portion of \cite[Proposition 2.9]{zbMATH05172034} by identifying vector/matrix bundles with principal bundles over unitary and projective unitary groups respectively (\Cref{res:mtrz.as.vect}\Cref{item:res:mtrz.as.vect:rec.upu}). There are other aspects to that result's finite-type characterization, quite apart from what the preceding theorem covers though; it turns out that the finite-type matrix bundles $\cA\xrightarrowdbl{}X$ are also precisely those for which a ``tensor neutralizer'' exists: $\cA\otimes \cB$ is trivial for some matrix bundle $\cB\xrightarrowdbl{}X$.

This prompts a foray into the issue of whether and to what extent matrix/vector bundles$_{\bU}$ can be trivialized in this fashion. Triviality admitting multiple interpretations in the context of equivariant bundles, we refer to its strongest incarnation as \emph{neutrality} (\Cref{def:ntrl}): the diagonal $\bU$-action on $X\times V$ for a representation $V$ (perhaps carrying a $\bU$-equivariant $C^*$-algebra structure, inner product, etc.). A sample result:

\begin{theorem}\label{th:which.cpct.lie.gps.loc}
  The following conditions on a compact Lie group $\bU$ are equivalent.

  \begin{enumerate}[(a),wide]
  \item\label{item:th:which.cpct.lie.gps.loc:alg}  For every normal $\bU$-space $X$, a locally trivial subhomogeneous $C^*$ bundle$_{\bU}$ $\cA\xrightarrowdbl{}X$
    \begin{enumerate}[(i),wide]
    \item\label{item:th:which.cpct.lie.gps.loc:alg:ft} is finite-type
    \item\label{item:th:which.cpct.lie.gps.loc:alg:emb} if and only if it embeds unitally into a neutral matrix (or subhomogeneous $C^*$) bundle$_{\bU}$ locally over a finite open cover by $\bU$-invariant sets;

    \item\label{item:th:which.cpct.lie.gps.loc:alg:tens} and when it is a matrix bundle, also if and only if there is a matrix bundle$_{\bU}$ $\cB\xrightarrowdbl{}X$ with $\cA\otimes \cB$ neutral, over a finite open cover by $\bU$-invariant sets.
    \end{enumerate}
  \item\label{item:th:which.cpct.lie.gps.loc:vec} For every compact-$T_2$-base locally trivial vector bundle$_{\bU}$ $\cE\xrightarrowdbl{}X$ there are, locally on $X$, non-zero bundles$_{\bU}$ $\cF$ with $\cE\otimes \cF$ neutral.  \qedhere
  \end{enumerate}
\end{theorem}

\emph{Local} neutrality is the guessable variation: the above conditions are required to hold only locally, and actions $\bU\circlearrowright X$ are deemed locally neutral if all vector bundles$_{\bU}$ $\cE\xrightarrowdbl{}X$ are so. This is an entirely orbit-driven property, and consequently representation-theoretic in nature:

\begin{proposition}\label{pr:loc.ntrl}
  The following conditions on a compact Lie group $\bU$ are equivalent.

  \begin{enumerate}[(a),wide]
  \item\label{item:pr:loc.ntrl:loc.ntrl} $\bU$ is locally neutral.

  \item\label{item:pr:loc.ntrl:mfld.ntrl} All smooth $\bU$-actions on compact $C^{\infty}$ manifolds are neutral.
    
  \item\label{item:pr:loc.ntrl:orb.ntrl} $\bU$ acts neutrally on all of its homogeneous spaces $\bU/\bH$ by closed subgroups.

  \item\label{item:pr:loc.ntrl:rep.ext} For every closed $\bH\le \bU$ and every representation $\bH\circlearrowright V$ there is another, $\bH\circlearrowright W$, with $V\otimes W$ restricted from $\bU$. 
  \end{enumerate}
\end{proposition}

Writing $\bU\overset{\circlearrowright}{/}\bH$ as shorthand for \Cref{pr:loc.ntrl}\Cref{item:pr:loc.ntrl:rep.ext}'s condition on $\bH\le \bU$, there are some of the (perhaps expected) permanence results:

\begin{theorem}\label{th:fin.ind.dscnt}
  \begin{enumerate}[(1),wide]
  \item\label{item:th:fin.ind.dscnt:fin.ind.l} If $\bK\le \bH\le \bU$ is a chain of compact Lie groups with finite-index $\bK\le \bH$ then
    \begin{equation*}
      \bU \overset{\circlearrowright}{/} \bK      
      \iff
      \bU \overset{\circlearrowright}{/} \bH.
    \end{equation*}

  \item\label{item:th:fin.ind.dscnt:fin.ind.r} If $\bK,\bH\le \bU$ are closed subgroups of a compact Lie group with finite-index $\bH\le \bU$ then
    \begin{equation*}
      \bU \overset{\circlearrowright}{/} \bK
      \iff
      \bU \overset{\circlearrowright}{/} \left(\bK\cap \bH\right)
      \iff
      \bH \overset{\circlearrowright}{/} \left(\bK\cap \bH\right).
    \end{equation*}

  \item\label{item:th:fin.ind.dscnt:fin.ind.r.glob} In particular, a finite-index $\bH\le \bU$ is locally $\cX$-neutral for a family of $\bU$-spaces/actions if and only if $\bU$ is. 
  \end{enumerate}
\end{theorem}

To conclude and come full circle back to flavors of finite-type characterizations, recall that \emph{virtually abelian} groups are those having finite-index abelian subgroups (and generally, \emph{virtually $\cP$} means having a finite-index subgroup with any given property $\cP$). As a consequence of the local-neutrality material in \Cref{se:ntrl} we have the following equivariant amplification of \cite[Theorem A]{2605.10752v1}. 

\begin{corollary}\label{cor:virt.ab.istrp}
  For a compact-Lie action $\bU\circlearrowright X$ with virtually abelian isotropy on a normal space, a locally trivial $C^*$ bundle$_{\bU}$ is finite-type if and only if it embeds unitally into neutral matrix bundles$_{\bU}$ locally over a finite open cover by $\bU$-invariant subsets.
\end{corollary}


\section{Extending bundles along equivariant compactifications}\label{se:ext.eqv.cpctf}

Bundles (principal, vector, etc.) \emph{equivariant} under appropriately compatible actions by a topological (here, exclusively compact Lie) group $\bU$ are referred to as \emph{bundles$_{\bU}$}. For the well-developed theory the reader can consult \cite[\S I.8]{td_transf-gp}, \cite[Chapter VII]{may_eq-hotop_1996} or \cite{MR885537} say, with more specific references given below as needed. 

\begin{remark}\label{re:triv}
  \emph{Triviality} is somewhat semantically ambiguous in the context of equivariant bundles:
  \begin{itemize}[wide]
  \item On the one hand, in language compatible with what \cite[post Lemma I.8.9 and Definition I.9.1]{td_transf-gp} term \emph{locally trivial} bundles$_{\bU}$ (whether principal or vector), one might mean those objects admitting bundle$_{\bU}$ morphisms to \emph{local objects} \cite[post Example I.8.8]{td_transf-gp}: bundles$_{\bU}$ over homogeneous spaces attached to closed subgroups $\bH\le \bU$. 

  \item On the other, for vector bundles$_{\bU}$ $\cE\xrightarrowdbl{} X$ in particular, triviality is initially construed in \cite[\S I.9]{td_transf-gp} more restrictively: lying in the image of the (\emph{symmetric monoidal} \cite[Definitions 8.1.7 and 8.1.12]{egno}) functor
    \begin{equation}\label{eq:reps.to.bdls}
      \Rep(\bU)
      \ni \left(\rho:\bU\circlearrowright V\right)
      \xmapsto{\quad}
      \rho_X
      =
      V_X
      :=
      X\times V
    \end{equation}
    on the category of finite-dimensional $\bU$-representations (with the diagonal $\bU$-action on $V_X$). This is indeed more restrictive, as it amounts to being a pullback through the $\bU$-space map $X\to \{*\}$. 
  \end{itemize}
  To avoid ambiguity, we refer to these latter bundles as \emph{neutral}. The term applies to richer structure: a neutral algebra-bundle, for instance, is one lying in the image of \Cref{eq:reps.to.bdls}'s lift to a functor between the respective categories of internal algebras.
\end{remark}

Some of the subsequent discussion relies on the notion of \emph{equivariant compactification} $X\to \beta_{\bU}X$ for a $\bU$-space $X$: the $X$ component of the \emph{unit} \cite[\S IV.1, post Theorem 1]{mcl_2e} of an adjunction, with $\beta_{\bU}$ being left adjoint to the inclusion functor
\begin{equation*}
  \left(\text{compact Hausdorff $\bU$-spaces}\right)
  =:
  \cat{Top}_{T_2,\kappa}^{\bU}
  \lhook\joinrel\xrightarrow{\quad}
  \cat{Top}_{T_{3\frac 12}}^{\bU}
  :=
  \left(\text{$T_{3\frac 12}$ $\bU$-spaces}\right);
\end{equation*}
that left adjoint exists provided $\bU$ is locally compact (Hausdorff) \cite[Proposition 3.1]{dvr-ex}.

\begin{remark}\label{re:cpctf.ladj}
  That left adjointness description of $\beta_{\bU}$ ensures some of the expected permanence properties hold: the canonical map
  \begin{equation*}
    \beta_{\bU}X/\bK
    \xrightarrow{\quad}
    \beta_{\bU/\bK}\left(X/\bK\right)
  \end{equation*}
  is a homeomorphism for instance, for normal subgroups $\bK\trianglelefteq \bU$ and $T_{3\frac 12}$ $\bU$-spaces $X$. This is proven in a more ad-hoc fashion as \cite[Theorem 4.9]{MR2166267}, but it suffices to observe that the upper and lower functor compositions
  \begin{equation*}
    \begin{tikzpicture}[>=stealth,auto,baseline=(current  bounding  box.center)]
      \path[anchor=base] 
      (0,0) node (l) {$\cat{Top}_{T_{3\frac 12}}^{\bU}$}
      +(3,.5) node (u) {$\cat{Top}_{T_2,\kappa}^{\bU}$}
      +(3,-.5) node (d) {$\cat{Top}_{T_{3\frac 12}}^{\bK}$}
      +(6,0) node (r) {$\cat{Top}_{T_2, \kappa}^{\bK}$}
      ;
      \draw[->] (l) to[bend left=6] node[pos=.5,auto] {$\scriptstyle \beta_{\bU}$} (u);
      \draw[->] (u) to[bend left=6] node[pos=.5,auto] {$\scriptstyle /\bK$} (r);
      \draw[->] (l) to[bend right=6] node[pos=.5,auto,swap] {$\scriptstyle /\bK$} (d);
      \draw[->] (d) to[bend right=6] node[pos=.5,auto,swap] {$\scriptstyle \beta_{\bU/\bK}$} (r);
    \end{tikzpicture}
  \end{equation*}
  are both left adjoint to the selfsame inclusion-restriction functor casting a compact $T_2$ $\bU/\bK$-space as a $T_{3\frac 12}$ $\bU$-space (for compositions of left adjoints are left adjoint \cite[Proposition 3.2.1]{brcx_hndbk-1}). 
\end{remark}

The same adjoint-composition principle drives also the following simple remark, perhaps apposite at this point (though not, strictly speaking, needed below). 

\begin{lemma}\label{le:ind.act}
  For any compact-group morphism $\bH\to \bU$ and $T_{3\frac 12}$ $\bH$-space $\bH\circlearrowright X$ the canonical $\bU$-equivariant map $\bU\times_{\bH}\beta_{\bH}X\to \beta_{\bU}X$ is a $\bU$-space isomorphism. 
\end{lemma}
\begin{proof}
  As in \Cref{re:cpctf.ladj}: the two rightward functorial compositions in
  \begin{equation*}
    \begin{tikzpicture}[>=stealth,auto,baseline=(current  bounding  box.center)]
      \path[anchor=base] 
      (0,0) node (l) {$\cat{Top}_{T_{3\frac 12}}^{\bH}$}
      +(3,.5) node (u) {$\cat{Top}_{T_2,\kappa}^{\bH}$}
      +(3,-.5) node (d) {$\cat{Top}_{T_{3\frac 12}}^{\bU}$}
      +(6,0) node (r) {$\cat{Top}_{T_2, \kappa}^{\bU}$}
      ;
      \draw[->] (l) to[bend left=6] node[pos=.5,auto] {$\scriptstyle \beta_{\bH}$} (u);
      \draw[->] (u) to[bend left=6] node[pos=.5,auto] {$\scriptstyle \bU\times_{\bH}$} (r);
      \draw[->] (l) to[bend right=6] node[pos=.5,auto,swap] {$\scriptstyle \bU\times_{\bH}$} (d);
      \draw[->] (d) to[bend right=6] node[pos=.5,auto,swap] {$\scriptstyle \beta_{\bU}$} (r);
    \end{tikzpicture}
  \end{equation*}
  are both left adjoint to the forgetful functor $\cat{Top}_{T_2, \kappa}^{\bU}\to\cat{Top}_{T_{3\frac 12}}^{\bH}$.
\end{proof}

\pf{th:ft.iff.ext.cpct}
\begin{th:ft.iff.ext.cpct}
  \begin{enumerate}[label={},wide]
  \item\textbf{\Cref{item:th:ft.iff.ext.cpct:ft} $\Rightarrow$ \Cref{item:th:ft.iff.ext.cpct:fin.cw}:} The finite-type assumption ensures the existence, for some finite open cover $\left(U_j\right)_j$ of $X$, of $\bU$-equivariant bundle maps $\cP|_{U_j}\to \left(E\xrightarrowdbl{} \bU/\bH\right)$: the local objects of \Cref{re:triv}. The total spaces of those local-object bundles are compact by \cite[Lemma I.8.9 and preceding discussion]{td_transf-gp} (and our compactness assumptions on $\bU$, $\bG$), hence the existence (as in the proof of \cite[Theorem I.8.12]{td_transf-gp}) of a bundle$_{\bU}$ map
    \begin{equation*}
      \left(\cP\xrightarrowdbl{\quad}X\right)
      \xrightarrow{\quad}
      E_1*\cdots *E_n
      ,\quad
      \left(\text{"$*$" meaning topological \emph{join} \cite[\S 14.4.3]{td_alg-top}}\right)
    \end{equation*}
    which recovers $\cP$ as a pullback. 

  \item\textbf{\Cref{item:th:ft.iff.ext.cpct:fin.cw} $\Rightarrow$ \Cref{item:th:ft.iff.ext.cpct:univ.cpct}:} Having realized $\cP$ as a pullback through a map $X\to Y$ to \emph{any} compact $T_2$ $\bU$-space, pull back intermediately to the maximal compactification (by its defining universal property) per the commutative diagram
    \begin{equation*}
      \begin{tikzpicture}[>=stealth,auto,baseline=(current  bounding  box.center)]
        \path[anchor=base] 
        (0,0) node (l) {$X$}
        +(2,.5) node (u) {$\beta_{\bU}X$}
        +(4,0) node (r) {$Y$}
        ;
        \draw[right hook->] (l) to[bend left=6] node[pos=.5,auto] {$\scriptstyle $} (u);
        \draw[->] (u) to[bend left=6] node[pos=.5,auto] {$\scriptstyle $} (r);
        \draw[->] (l) to[bend right=6] node[pos=.5,auto,swap] {$\scriptstyle $} (r);
      \end{tikzpicture}
    \end{equation*}    
    
  \item\textbf{\Cref{item:th:ft.iff.ext.cpct:univ.cpct} $\Rightarrow$ \Cref{item:th:ft.iff.ext.cpct:cpct} $\Rightarrow$ \Cref{item:th:ft.iff.ext.cpct:ft}:} The first is immediate and the second follows from the fact that being of finite type is preserved by pullbacks and automatic for compact base spaces.

  \item\textbf{\Cref{item:th:ft.iff.ext.cpct:ft} $\Rightarrow$ \Cref{item:th:ft.iff.ext.cpct:all.ind} and \Cref{item:th:ft.iff.ext.cpct:all.rep} $\Rightarrow$ \Cref{item:th:ft.iff.ext.cpct:some.rep}} are immediate.

  \item\textbf{\Cref{item:th:ft.iff.ext.cpct:all.ind} $\Rightarrow$ \Cref{item:th:ft.iff.ext.cpct:all.rep} and \Cref{item:th:ft.iff.ext.cpct:some.rep} $\Leftrightarrow$ \Cref{item:th:ft.iff.ext.cpct:some.ind}} are instances of the back-and-forth passage \cite[Assertion 18.3.4]{hjjm_bdle} between principal $U(n)$- ($O(n)$-)bundles and rank-$n$ complex (respectively real) vector bundles.

    This reduces the entire suite of mutual equivalences to only the concluding next step. 

  \item\textbf{\Cref{item:th:ft.iff.ext.cpct:some.ind} $\Rightarrow$ \Cref{item:th:ft.iff.ext.cpct:ft}:} Assume $\bG\le \bK$ embedded $\bU$-equivariantly. The induced bundle $\cP\times_{\bG}\bK$ extends across the equivariant compactification $X\lhook\joinrel\to \beta_{\bU}X$ by the already-established equivalence \Cref{item:th:ft.iff.ext.cpct:ft} $\Leftrightarrow$ \Cref{item:th:ft.iff.ext.cpct:univ.cpct}. Its being both finite-type and reducible to structure group $\bG$ over $X$, a trivialization over a finite cover $\cU=\left(U_i\right)_{i=1}^n$ of $\beta_{\bU} X$ provides
    \begin{itemize}
    \item $\bU$-equivariant \emph{transition functions}
      \begin{equation*}
        U_i\cap U_j
        \xrightarrow{\quad g_{ij}\quad}
        \bK;
      \end{equation*}
    \item along with equivariant \emph{gauge transformations}
      \begin{equation*}
        X\cap U_i
        =:
        U_{X,i}\xrightarrow{\lambda_i}\bG
      \end{equation*}
      with
      \begin{equation}\label{eq:land.in.g}
        \forall\left(1\le i,j\le n\right)
        \bigg(\left(\lambda_i^{-1}g_{ij}\lambda_j\right)\left(U_{X,i}\cap U_{X,j}\right)\subseteq \bG\bigg):
      \end{equation}
    \end{itemize}
    this is the equivariant version of the familiar transition-function formalism for principal bundles generally \cite[\S 5.2]{hus_fib} and applicable to structure-group reduction in particular \cite[Theorem 6.4.1]{hus_fib}. The $\lambda_i$ constitute a single continuous $\bU$-map $X\to \bG^{*n}$ which extends to all of $\beta_{\bU}X$ by the latter's universal property. \Cref{eq:land.in.g} holds over $U_i\cap U_j$ by continuity, hence the reduction of $\cP\times_{\bG}\bK$'s extension to $\beta_{\bU}X$ to structure group $\bG$.  \qedhere
  \end{enumerate}
\end{th:ft.iff.ext.cpct}

\begin{remarks}\label{res:mtrz.as.vect}
  \begin{enumerate}[(1),wide]
  \item\label{item:res:mtrz.as.vect:rec.upu} To recover the matrix/vector-bundle dichotomy operative in \cite[Proposition 2.9]{zbMATH05172034} in the context of \Cref{th:ft.iff.ext.cpct} specialize the compact-Lie embedding $\bG\le \bK$ to the faithful representation $PU(n)\le U(n^2)$ resulting from the projective unitary group's conjugation action on the $n\times n$-matrix algebra. This takes for granted the usual \cite[Assertion 18.3.4]{hjjm_bdle} conflation of vector (matrix) bundles with (projective) unitary principal bundles.

  \item\label{item:res:mtrz.as.vect:plb.cpct} Pulling back along maps to compact spaces (which featured in the proof of \Cref{th:ft.iff.ext.cpct}) suggests a slight simplification of the proof of \cite[Lemma 2.7]{zbMATH05172034}, to the effect that vector bundles $\cE\xrightarrowdbl{}X$ with compact Hausdorff base admit ``multiplicative complements'': non-zero vector bundles $\cF$ with $\cE\otimes \cF$ trivial. \cite[Lemma 12]{zbMATH03329610} covers the connected, finite-complex-base case, and the former result leverages this by expressing arbitrary compact spaces as \emph{cofiltered limits} \spr{04AY} of finite complexes.
    
    One can dispense with this limiting aspect by observing (along with the fact that $\cF$'s existence pulls back along maps) that arbitrary compact-base vector bundles$_{\bU}$ pull back along equivariant maps to Grassmannians of finite-dimensional $\bU$-representations \cite[(ii) post Theorem 2.5]{seg}. In \cite[Lemma 2.7]{zbMATH05172034} in its original form and whatever equivariant versions thereof one might be interested in, then,
    \begin{itemize}[wide]
    \item the base space can be assumed a Grassmannian of a $\bU$-representation (so in particular a compact smooth $\bU$-manifold);

    \item and the bundle can be assumed tautological over that Grassmannian: the bundle of $q$-planes over the $q$-plane Grassmannian $\bG(q,V)$ whose fiber over $W\le V$ ($\dim W=q$) is precisely $W$ (cf. \cite[Example VII.3.1]{may_eq-hotop_1996}).
    \end{itemize}

  \item\label{item:res:mtrz.as.vect:good.gp} By the same token, in a direct equivariant version of \cite[Lemma 12]{zbMATH03329610} the acting group $\bU$ can be assumed unitary (orthogonal if working with real bundles, etc.). Equally (to the extent that this is convenient), said group can be assumed simple either Lie-theoretically (connected with simple Lie algebra) or algebraically (no non-obvious normal subgroups): a faithful representation $\bU\lhook\joinrel\to U(n)$ can always be extended to an embedding into a special or projective unitary group.

  \item\label{item:res:mtrz.as.vect:cech.cont} The cofiltered-limit argument recalled briefly in item \Cref{item:res:mtrz.as.vect:plb.cpct} above also relies implicitly on the continuity of the contravariant functor
    \begin{equation*}
      \left(\text{compact $T_2$ spaces}\right)
      \ni
      X
      \xmapsto{\quad}
      \left(\text{iso-classes of rank-$q$ vector bundles/$X$}\right)
    \end{equation*}
    with respect to such limits. This follows, say, from \cite[Theorem 4 and Remark 4.11]{MR234450}.
  \end{enumerate}  
\end{remarks}

As actions $\bU\circlearrowright X$ with finite \emph{isotropy groups} 
\begin{equation*}
  \bU_x:=\left\{g\in \bU\ :\ gx=x\right\}
  ,\quad
  x\in X
\end{equation*}
will feature in the sequel, note that the actions of \Cref{th:ft.iff.ext.cpct}, supporting finite-type bundles$_{\bU}$, need not extend with finite isotropy to $\beta_{\bU}X$.

\begin{example}\label{ex:no.fin.istrp}
  By the very definition \cite[\S VII.1]{may_eq-hotop_1996} of a trivial principal $\bG$-bundle$_{\bU}$, \emph{any} $\bU$-space supports one such. Given $\bU\circlearrowright X$ with finite isotropy groups with unbounded order (e.g. the circle $\bU:=\bS^1$ acting on $X:=\coprod_{n\in \bZ_{>0}}\bS^1$ via the respective quotient maps $z\mapsto z^n$), the easily proven \emph{upper semicontinuity} \cite[Definition 7.1.1]{kt_corresp} of
  \begin{equation*}
    \beta_{\bU}X\ni y
    \xmapsto{\quad}
    \bU_y
    \in
    \left\{\text{closed subgroups of $\bU$}\right\}
  \end{equation*}
  in the \emph{Vietoris topology} ensures infinite isotropy at some $y\in \beta_{\bU}X\setminus X$.
\end{example}

\Cref{ex:no.fin.istrp} suggests the natural question of whether assuming \emph{bounded-order} isotropy might remedy the issue; it does not (even for \emph{trivial} isotropy, i.e. free actions), as the following remark, very much also in the spirit of \cite[Example 3.8]{anton_equiv-emb}, shows.

\begin{lemma}\label{le:s1.act.on.hilb.sph}
  Let $\ell^2_{1}$ be the unit sphere in the countable-dimensional Hilbert space, equipped with the free scaling action $\bS^1\circlearrowright \ell^2_1$.
   
  The universal compactification $\beta_{\bS^1}\ell^2_1$ has at least one $\bS^1$-fixed point, hence so do all equivariant compactifications of $\ell^2_1$.
\end{lemma}
\begin{proof}
  This is an adaptation of \cite[Example 3.6]{MR2166267}, where the acting group is $\bZ/2$ and the Hilbert space involved is real (the conclusion there being that the universal action on $\beta_{\bZ/2}\ell^2_{1}$ must have fixed points).
  
  It suffices to argue, by the Vietoris upper semicontinuity noted in \Cref{ex:no.fin.istrp}, that no equivariant compactification can have bounded isotropy (equivalently, $\beta_{\bS^1}\ell^2_1$ does not). 
  \begin{itemize}[wide]
  \item Were there a bounded-isotropy compactification $\ol{X}$ of $X:=\ell^2_{1}$, the isotropy groups $\bS^1_y$, $y\in \ol{X}$ would all be contained in a cyclic subgroup of $\bS^1$ of order $n$, say.

  \item For $\gcd(m,n)=1$, $\bZ/m$ would then act freely on $\ol{X}$.

  \item This would produce a $\bZ/m$-equivariant map
    \begin{equation*}
      X
      \lhook\joinrel\xrightarrow{\quad}
      \ol{X}
      \to
      \left(\bZ/m\right)^{* (k+1)}
      \lhook\joinrel\xrightarrow{\quad}
      \left(\bS^1\right)^{* (k+1)}
      \cong
      \bS^{2k+1}
    \end{equation*}
    for some $k\in \bZ_{\ge 0}$ (``$*$'' denoting joins): locally by \emph{tube} existence \cite[Theorem I.5.7]{td_transf-gp} and thence globally by the usual partition-of-unity argument \cite[proof of Theorem I.8.12]{td_transf-gp}.

  \item Arbitrarily large odd spheres $\bS^{2q+1}$ being $\bS^1$-equivariantly embeddable in $X$, this produces $\bZ/m$-equivariant maps $\bS^{2q+1}\to \bS^{2k+1}$, $q> k$ which contradict one version \cite[Theorem, p.65]{MR711043} of the \emph{Borsuk-Ulam theorem}.
  \end{itemize}
\end{proof}

A somewhat different strategy will prove more: \emph{all} infinite compact Lie groups admit free actions $\bU\circlearrowright X$ on metrizable spaces with infinite-isotropy points in $\beta_{\bU}X$. Rather than unit spheres of Hilbert-space representations, we can take for $X$ the total space of \emph{Milnor's model} \cite[\S 14.4.3]{td_alg-top}
\begin{equation*}
  E\bU
  :=
  \bigcup_n \left(E_n\bU:=\bU^{*(n+1)}\right)
\end{equation*}
for the universal principal $\bU$-bundle.  

\begin{proposition}\label{pr:all.inf.cpct.lie}
  For any compact Lie group $\bU$ the action $\bU\circlearrowright \beta_{\bU}E\bU$ has points fixed by any one maximal torus $\bT\le \bU_0$ of $\bU$'s identity component. 
\end{proposition}
\begin{proof}
  $\bT$ is a Vietoris limit
  \begin{equation*}
    \bT=\lim_{\alpha}\bH_{\alpha}
    ,\quad
    \text{finite $p$-groups $\bH_{\alpha}\le \bT$}
  \end{equation*}
  for various primes $p$. Now, the $\bH_{\alpha}$-action $\bH_{\alpha}\circlearrowright \beta_{\bH_{\alpha}}E\bH_{\alpha}$ has fixed points \cite[Theorem 6.1]{MR1911520}, and hence so does the $\bH_{\alpha}$-action inherited via $\bH_{\alpha}\le \bT\le \bU$ on the closure of $E\bH_{\alpha}\subseteq E\bU$ in $\beta_{\bU}E\bU$. The limit of a convergent subnet of a net consisting of such respective $\bH_{\alpha}$-fixed points $x_{\alpha}\in \beta_{\bU}E\bU$ will be $\bT$-fixed. 
\end{proof}

\section{Trivial embeddability and neutral actions}\label{se:ntrl}

In addition to the (essentially two) alternative characterizations of the finite-type property touched upon by \Cref{th:ft.iff.ext.cpct}, involving respectively extensibility along compactifications and finite type after structure-group extension, \cite[Proposition 2.9]{zbMATH05172034} also equates finite type, for matrix bundles $\cA\xrightarrowdbl{}X$, with unital embeddability into trivial matrix bundles. This has \cite[Lemma 2.7]{zbMATH05172034} as an auxiliary, building on \cite[Lemma 12]{zbMATH03329610} in the manner sketched in \Cref{res:mtrz.as.vect}\Cref{item:res:mtrz.as.vect:plb.cpct}.

With a view towards extending such results equivariantly, we record here how trivializability of equivariant bundles by tensoring is not only sufficient, but also necessary for the algebra-bundle analogue (robustly to which class of compact Lie groups one considers as a basis for equivariantization). 

\begin{theorem}\label{th:which.cpct.lie.gps}
  The following conditions on a compact Lie group $\bU$ are equivalent.

  \begin{enumerate}[(a),wide]
  \item\label{item:th:which.cpct.lie.gps:alg}  For every normal $\bU$-space $X$, a locally trivial subhomogeneous $C^*$ bundle$_{\bU}$ $\cA\xrightarrowdbl{}X$
    \begin{enumerate}[(i),wide]
    \item\label{item:th:which.cpct.lie.gps:alg:ft} is finite-type
    \item\label{item:th:which.cpct.lie.gps:alg:emb} if and only if it embeds unitally into a neutral matrix (or subhomogeneous $C^*$) bundle$_{\bU}$;

    \item\label{item:th:which.cpct.lie.gps:alg:tens} and when it is a matrix bundle, also if and only if there is a matrix bundle$_{\bU}$ $\cB\xrightarrowdbl{}X$ with $\cA\otimes \cB$ neutral.  
    \end{enumerate}
  \item\label{item:th:which.cpct.lie.gps:vec} For every compact-$T_2$-base locally trivial vector bundle$_{\bU}$ $\cE\xrightarrowdbl{}X$ there is a non-zero such, $\cF$, with $\cE\otimes \cF$ neutral. 
  \end{enumerate}
\end{theorem}
\begin{proof}
  \begin{enumerate}[label={},wide]
  \item\textbf{\Cref{item:th:which.cpct.lie.gps:alg} $\Rightarrow$ \Cref{item:th:which.cpct.lie.gps:vec}} If compact-base $\cE$ has $\left(\cE\otimes \cE^*\right)\otimes \cB$ neutral as a matrix bundle$_{\bU}$ (so by necessity also as a plain vector bundle$_{\bU}$), then tensoring with $\cE^*\otimes \cB$ neutralizes $\cE$.

  \item\textbf{\Cref{item:th:which.cpct.lie.gps:vec} $\Rightarrow$ \Cref{item:th:which.cpct.lie.gps:alg}} Given the stronger version of \Cref{item:th:which.cpct.lie.gps:alg}\Cref{item:th:which.cpct.lie.gps:alg:emb} in the form of an embedding $\cA\lhook\joinrel\to \cM$ for a neutral matrix bundle $\cM$, simply take for $\cB$ the bundle $\cA'\le \cM$ of \emph{commutants}:
    \begin{equation*}
      \cB_x
      :=
      \cA'_x
      :=
      \left\{b\in \cM_x\ :\ ba=ab,\ \forall a\in \cA_x\right\}. 
    \end{equation*}
    We can thus focus on \Cref{item:th:which.cpct.lie.gps:alg}\Cref{item:th:which.cpct.lie.gps:alg:ft} $\Leftrightarrow$ \Cref{item:th:which.cpct.lie.gps:alg}\Cref{item:th:which.cpct.lie.gps:alg:emb} alone, with $C^*$ and matrix-bundle embeddability (the two versions of \Cref{item:th:which.cpct.lie.gps:alg}\Cref{item:th:which.cpct.lie.gps:alg:emb}) equivalent \cite[Theorem 2.4, (b) $\Leftrightarrow$ (c)]{2605.10752v1}.

    Being finite-type is inherited by subalgebra bundles \cite[Theorem 2.4 (e) $\Rightarrow$ (a)]{2605.10752v1}, so it is \Cref{item:th:which.cpct.lie.gps:alg}\Cref{item:th:which.cpct.lie.gps:alg:ft} $\Rightarrow$ \Cref{item:th:which.cpct.lie.gps:alg}\Cref{item:th:which.cpct.lie.gps:alg:emb} that carries the substance of the claims. As \Cref{th:ft.iff.ext.cpct} (applied to principal bundles over unitary or projective unitary groups) furthermore effects the transition from $T_4$ to compact Hausdorff, we can assume the base compact for the duration.
    
    Assuming $\cA$ of finite type even only as a vector bundle$_{\bU}$ (a condition equivalent to the hypothesis by the selfsame \cite[Theorem 2.4 (e) $\Rightarrow$ (a)]{2605.10752v1} and in any case formally weaker), $\cA$ operates on itself as a vector bundle by left multiplication. It then also operates on a \emph{neutral} vector bundle$_{\bU}$ of the form
    \begin{equation*}
      \cG:= \cA\otimes \cF
      ,\quad
      \cF\text{ a vector bundle$_{\bU}$ (as we are assuming \Cref{item:th:which.cpct.lie.gps:vec})},
    \end{equation*}
    so embeds unitally in the neutral matrix bundle $\cE nd~\cG:=\cG\otimes \cG^*$.
  \end{enumerate}
\end{proof}

Isolating the properties marked out by \Cref{th:which.cpct.lie.gps} will provide streamlined language. 

\begin{definition}\label{def:ntrl}
  \begin{enumerate}[(1),wide]
  \item\label{item:def:ntrl:ntrl} An action $\alpha:\bU\circlearrowright X$ of a compact (virtually always assumed Lie) group on a normal space is \emph{neutral} if all finite-type bundles$_{\bU}$ on $X$ can be neutralized by tensoring with other (non-zero) such. Equivalently (\Cref{th:ft.iff.ext.cpct}\Cref{item:th:ft.iff.ext.cpct:ft} $\Leftrightarrow$ \Cref{item:th:ft.iff.ext.cpct:univ.cpct}), the induced action $\bU\circlearrowright \beta_{\bU}X$ on the universal equivariant compactification is neutral. 
    
  \item\label{item:def:ntrl:loc.ntrl} $\alpha$ is \emph{locally neutral} if $X$ admits a finite cover by $\bU$-invariant open sets over which the action is neutral in the sense of \Cref{item:def:ntrl:ntrl}.

    The notions and language effectively quantify neutrality universally over bundles, having fixed the underlying action. One can also quantify over actions.

  \item\label{item:def:ntrl:gp.ntrl} $\bU$ itself is (\emph{locally}) \emph{neutral} if all of its actions on normal (equivalently, compact Hausdorff) spaces are. More generally, $\bU$ is \emph{(locally) $\cX$-neutral} for a class $\cX$ of spaces if its actions on spaces in $\cX$ are all (respectively locally) neutral. The same language applies if $\cX$ is a class of actions rather than spaces. 
  \end{enumerate}
\end{definition}

\Cref{pr:loc.ntrl} characterizes local neutrality.

\pf{pr:loc.ntrl}
\begin{pr:loc.ntrl}
  \begin{enumerate}[label={},wide]
  \item\textbf{\Cref{item:pr:loc.ntrl:loc.ntrl} $\Rightarrow$ \Cref{item:pr:loc.ntrl:mfld.ntrl} $\Rightarrow$ \Cref{item:pr:loc.ntrl:orb.ntrl}} are immediate on purely formal grounds.     
    
  \item\textbf{\Cref{item:pr:loc.ntrl:orb.ntrl} $\Rightarrow$ \Cref{item:pr:loc.ntrl:mfld.ntrl}:} This follows from the fact \cite[Theorem I.5.6]{td_transf-gp} that arbitrary orbits in compact smooth $\bU$-manifolds have \emph{tube} neighborhoods that deformation-retract equivariantly onto said orbits. 
    
  \item\textbf{\Cref{item:pr:loc.ntrl:orb.ntrl} $\Leftrightarrow$ \Cref{item:pr:loc.ntrl:rep.ext}} follows from the identification \cite[p.130, (c)]{seg} between vector bundles$_{\bU}$ on $\bU/\bH$ and $\bH$-representations. 

  \item\textbf{\Cref{item:pr:loc.ntrl:mfld.ntrl} $\Rightarrow$ \Cref{item:pr:loc.ntrl:loc.ntrl}:} Finite-type vector bundles$_{\bU}$ pull back along maps to compact smooth $\bU$-manifolds (\Cref{res:mtrz.as.vect}\Cref{item:res:mtrz.as.vect:plb.cpct} + \Cref{th:ft.iff.ext.cpct}) and being neutralized by tensoring survives pullbacks.  \qedhere
  \end{enumerate}  
\end{pr:loc.ntrl}

\Cref{th:which.cpct.lie.gps} can now be regarded as a global version of \Cref{th:which.cpct.lie.gps.loc}, whose proof we omit given the earlier argument's immediate adaptability.

\Cref{pr:loc.ntrl.fin} gathers some observations on local neutrality. Given the ubiquity of neutral actions on homogeneous spaces, we indicate $\bU$'s acting neutrally on $\bU/\bH$ by $\bU\overset{\circlearrowright}{/}\bH$. Throughout the sequel, $\Ad_{\bullet}$ denotes conjugation actions (e.g. of groups on normal subgroups, or on isomorphism classes of representations thereof). 

\begin{proposition}\label{pr:loc.ntrl.fin}
  \begin{enumerate}[(1),wide]
  \item\label{item:pr:loc.ntrl.fin:trnstv} For a chain of closed $\bK\le \bH\le \bU$ Lie subgroups in a compact Lie group we have
    \begin{equation*}
      \bU\overset{\circlearrowright}{/}\bH
      \ \wedge \ 
      \bH\overset{\circlearrowright}{/}\bK
      \xRightarrow{\quad}
      \bU\overset{\circlearrowright}{/}\bK
      \quad\text{and}\quad
      \bU\overset{\circlearrowright}{/}\bK
      \xRightarrow{\quad}
      \bH\overset{\circlearrowright}{/}\bK.
    \end{equation*}
    
  \item\label{item:pr:loc.ntrl.fin:fin.ind} Compact Lie groups act neutrally on their finite homogeneous spaces.

  \item\label{item:pr:loc.ntrl.fin:fin.gp} In particular, finite groups are locally neutral. 
  \end{enumerate}
\end{proposition}
\begin{proof}
  \Cref{item:pr:loc.ntrl.fin:trnstv} is a simple unpacking of the definitions and \Cref{item:pr:loc.ntrl.fin:fin.ind} $\Rightarrow$ \Cref{item:pr:loc.ntrl.fin:fin.gp} is tautological; we thus only address \Cref{item:pr:loc.ntrl.fin:fin.ind}.
  
  In the representation-theoretic form of \Cref{pr:loc.ntrl}'s \Cref{item:pr:loc.ntrl:rep.ext}, the claim is that representations of finite-index subgroups $\bH\le \bU$ can be tensored with others such, non-zero, so as to produce $\bU$-representations. Observe first that this does hold for \emph{normal} $\bH\trianglelefteq \bU$. Indeed, the $\bH$-representations restricted from $\bU$ are sums of
  \begin{equation*}
    \left(\bigoplus \cO_{\rho}\right)^{\oplus n_{\rho}}
    ,\quad
    \text{some fixed $\rho$-dependent }n_{\rho}\in \bZ_{\ge 0},
  \end{equation*}
  where $\cO_{\rho}$ is the orbit in $\cat{Irr}(\bH)$ of the irreducible $\bH$-representation $\rho$ under $\bU$-conjugation: this is \emph{Clifford's theorem} \cite[Theorem 6.2]{is_char_1976} in its compact-group form (also implicit in \cite[Theorem 3.11 and preceding discussion]{mack-unit}, say). Now, an arbitrary $\bH$-representation can first be tensored with all of its $\bU$-conjugates, and then further tensored with a trivial representation whose dimension is divisible by sufficiently many $n_{\rho}$.

  For the general case, consider a finite-index subgroup $\bH\le \bU$ in a compact Lie group together with a finite-index normal subgroup $\bK\trianglelefteq \bU$ contained in $\bH$ (e.g. $\bK:=\bigcap_{g\in \bU} g \bH g^{-1}$, automatically of finite index, or simply the identity connected component $\bK:=\bU_0$). If $\rho\in \Rep(\bH)$ is such that $\Ad_g \left(\rho|_{\bK}\right)$, $g\in \bU$ are all mutually isomorphic, some multiples of
  \begin{equation}\label{eq:2inds}
    \rho\otimes \mathrm{Ind}^{\bH}_{\bK}\mathbf{1}
    \cong
    \mathrm{Ind}^{\bH}_{\bK}\rho|_{\bK}
    \quad\text{and}\quad
    \left(\mathrm{Ind}^{\bU}_{\bK}\rho|_{\bK}\right)|_{\bH}
    \cong
    \bigoplus_{s\in \bK\backslash \bU/\bH=\bU/\bH}\mathrm{Ind}_{\bK}^{\bH}\Ad_s\rho|_{\bK}
  \end{equation}
  will coincide, with the second isomorphism being an instance of Mackey's restriction-induction \cite[Theorem 7.1]{mack-ind-1} (\cite[Proposition 22]{serre_rep_1977} for finite groups). Arbitrary $\rho$ can first be replaced by
  \begin{equation*}
    \bigotimes_{s\in \bU/\bK}\mathrm{Ind}_{\bK}^{\bH}\Ad_s \rho|_{\bK},
  \end{equation*}
  (with $\rho$ as a tensor factor by \Cref{eq:2inds}'s left-hand side) ensuring the aforementioned invariance property.
\end{proof}

\begin{remark}\label{re:fin.loc.ntrl}
  The conclusion of \Cref{pr:loc.ntrl.fin}\Cref{item:pr:loc.ntrl.fin:fin.gp} is easily obtainable directly from the absorption property of a finite group's regular representation $\rho_{H}$: for all $\rho\in \Rep(\bH)$, $\rho_H\otimes \rho\cong \rho_H^{\oplus \dim \rho}$ (e.g. by its character's vanishing \cite[Proposition 5]{serre_rep_1977} on all non-trivial elements). For $\bH\le \bU$ (both finite), then, $\left(\rho_{\bU}|_{\bH}\right)\otimes \rho\cong \left(\rho_{\bU}|_{\bH}\right)^{\oplus [\bU:\bH] \cdot \dim \rho}$ for any $\rho\in\Rep(\bH)$. 
\end{remark}

Further partial results on local neutrality include the following.

\begin{proposition}\label{pr:fin.tor.loc.ntrl}
  \begin{enumerate}[(1),wide]
  \item\label{item:pr:fin.tor.loc.ntrl:ab} Abelian compact Lie groups are locally neutral. 

  \item\label{item:pr:fin.tor.loc.ntrl:subtor.fin} Compact Lie groups $\bU$ act neutrally on homogeneous spaces $\bU/\bH$ with $\bH$ either finite or \emph{sub-toral}, i.e. contained in a torus. 
  \end{enumerate}  
\end{proposition}
\begin{proof}
  \begin{enumerate}[label={},wide]
  \item\textbf{$\Cref{item:pr:fin.tor.loc.ntrl:ab}$} Representations of compact abelian groups simply extend from subgroups: \emph{Pontryagin duality} \cite[\S 1.7.3]{rud_lc} turns the claim into that of $\bZ^d$-domain morphisms in the category $\cat{Ab}$ of (discrete) abelian groups lifting along epimorphisms, which follows from $\bZ^d$'s projectivity in $\cat{Ab}$.
    
  \item\textbf{$\Cref{item:pr:fin.tor.loc.ntrl:subtor.fin}$: sub-toral.} A repeated application of \Cref{pr:loc.ntrl.fin}\Cref{item:pr:loc.ntrl.fin:trnstv} to the two contiguous 3-term chains in
      \begin{equation*}
        \bH\le \bT\overset{\text{maximal torus}}{\le} \bU_0\le \bU
        ,\quad
        \bU_0:=\text{identity component}
      \end{equation*}
      reduces the problem to $\bT\overset{\circlearrowright}{/}\bH$ and $\bU_0\overset{\circlearrowright}{/}\bT$, with the first settled by \Cref{item:pr:fin.tor.loc.ntrl:ab}. It thus remains to handle maximal tori in compact connected $\bU$. Now, for any $\bT$-representation $\rho$
      \begin{itemize}[wide]

      \item the product $\wt{\rho}$ of $\rho$'s conjugates under the \emph{Weyl group} $W=W(\bT,\bU)$ \cite[Definition 4.1.3]{btd_lie_1995} is a $\bT$-representation restricted from a \emph{virtual character} of $\bG$, per the usual identification $R(\bU)\cong R(\bT)^{W}$ \cite[Proposition VI.2.1]{btd_lie_1995} of representation rings;

      \item whence the extensibility of $\wt{\rho}\otimes \psi|_{\bT}$ to $\bU$ for some $\psi\in \Rep(\bU)$ by \cite[Corollaire 2]{MR838399}.
      \end{itemize}

    \item\textbf{$\Cref{item:pr:fin.tor.loc.ntrl:subtor.fin}$: finite.} The regular representation's absorbing property noted in \Cref{re:fin.loc.ntrl} reduces the problem to proving some multiple thereof extensible to $\bU$; this, in different language, is precisely what the proof of \cite[Theorem 9]{MR893156} delivers (for the $\bH$-character appearing there, vanishing on all non-trivial elements of the finite group $\bH$, is a rational multiple of the regular representation's character). 
  \end{enumerate}  
\end{proof}

\Cref{th:fin.ind.dscnt} is very much in the same spirit of examining permanence properties for neutral homogeneous actions. It will be convenient to have shorthand language/notation for describing representations with coinciding finite powers: $\rho^{\oplus m}\cong \rho'^{\oplus n}$ for some $m,n\in \bZ_{>0}$. In that case we call $\rho$ and $\rho'$ \emph{commensurable}, and write $\rho\sim \rho'$. 

\pf{th:fin.ind.dscnt}
\begin{th:fin.ind.dscnt}
  \Cref{item:th:fin.ind.dscnt:fin.ind.r} plainly implies \Cref{item:th:fin.ind.dscnt:fin.ind.r.glob}.
  \begin{enumerate}[label={},wide]
  \item\textbf{\Cref{item:th:fin.ind.dscnt:fin.ind.l}:} We have $\bH\overset{\circlearrowright}{/}\bK$ in any case by \Cref{pr:loc.ntrl.fin}\Cref{item:pr:loc.ntrl.fin:fin.ind}, hence also ($\Leftarrow$) by \Cref{pr:loc.ntrl.fin}\Cref{item:pr:loc.ntrl.fin:trnstv}. 

    For ($\Rightarrow$), we can assume $\bK\trianglelefteq \bH$ (normal) by the same device as that employed in the proof of \Cref{pr:loc.ntrl.fin}\Cref{item:pr:loc.ntrl.fin:fin.ind}: 
    \begin{equation*}
      \bK\overset{\circlearrowright}{/}\left(\wt{\bK}:=\bigcap_{g\in \bH}\Ad_g \bK\right)
      \xRightarrow{\quad\text{\Cref{pr:loc.ntrl.fin}\Cref{item:pr:loc.ntrl.fin:trnstv}}\quad}
      \bU\overset{\circlearrowright}{/}\wt{\bK},
    \end{equation*}
    so $\wt{\bK}\trianglelefteq \bH$ can be substituted for $\bK$.

    The restriction $\rho|_{\bK}$ of an arbitrary $\rho\in \Rep(\bH)$ by assumption admits a complementary factor $\psi\in \Rep(\bK)$ with $\psi\otimes \rho|_{\bK}\in \Rep(\bU)|_{\bK}$, so that
    \begin{equation}\label{eq:can.lift}
      \exists\left(\theta\in \Rep(\bU)\right)
      \left(
        \bigotimes_{g\in \bH/\bK}\left(\left(\Ad_g \psi\right)\otimes \rho|_{\bK}\right)
        =
        \left(\rho|_{\bK}\right)^{\otimes |\bH/\bK|} \otimes
        \bigotimes_{g\in \bH/\bK}\Ad_g \psi
        \cong
        \theta|_{\bK}
      \right).
    \end{equation}
    On the one hand
    \begin{equation*}
      \bigotimes_{g\in \bH/\bK}\Ad_g \psi
      \sim
      \mathrm{Ind}_{\bK}^{\bH}\psi
      \quad
      \text{\cite[Theorem 7.1]{mack-ind-1} again}
    \end{equation*}
    so after implementing whatever necessary positive-integer scaling and appropriate tensorand regrouping/absorption \Cref{eq:can.lift}'s conclusion streamlines to
    \begin{equation}\label{eq:resk3}
      \left(\rho|_{\bK}\right)
      \otimes
      \left(\rho'|_{\bK}\right)
      \cong
      \theta|_{\bK}
      ,\quad
      \begin{aligned}
        \rho,\rho'&\in \Rep(\bH)\\
        \theta&\in\Rep(\bU)
      \end{aligned}
      .
    \end{equation}
    Further substituting
    \begin{equation*}
      \rho
      \xmapsto{\quad}
      \rho\otimes \eta|_{\bH}
      ,\quad
      \theta
      \xmapsto{\quad}
      \theta\otimes \eta      
    \end{equation*}
    for some representation $\eta\in \Rep(\bU)$ whose character vanishes on the set difference $\bH-\bK$ (always achievable \cite[proof of Proposition 11]{MR893156}), \Cref{eq:resk3} will be valid with $\bH$-restrictions in place of $|_{\bK}$: $\rho\otimes\rho'\cong \theta|_{\bH}$, for their characters agree on both $\bK$ \Cref{eq:resk3} and $\bH-\bK$ (where they vanish). This is the desired conclusion.
    
  \item\textbf{\Cref{item:th:fin.ind.dscnt:fin.ind.r}:} \Cref{item:th:fin.ind.dscnt:fin.ind.l} applied to $\bK\cap \bH\le \bK\le \bU$ (of which the left-hand embedding is finite-index) first yields
    \begin{equation*}
      \bU \overset{\circlearrowright}{/} \bK
      \iff
      \bU \overset{\circlearrowright}{/} \left(\bK\cap \bH\right).
    \end{equation*}
    This is further equivalent to $\bH \overset{\circlearrowright}{/} \left(\bK\cap \bH\right)$ by \Cref{pr:loc.ntrl.fin}\Cref{item:pr:loc.ntrl.fin:fin.ind}, ensuring that $\bU\overset{\circlearrowright}{/}\bH$, together with \Cref{pr:loc.ntrl.fin}\Cref{item:pr:loc.ntrl.fin:trnstv}.  \qedhere
  \end{enumerate}
\end{th:fin.ind.dscnt}

We record a consequence of the foregoing discussion. 

\begin{corollary}\label{cor:vrt.ab}
  \begin{enumerate}[(1),wide]
  \item Compact Lie groups act neutrally on homogeneous spaces with virtually abelian isotropy.

  \item Virtually abelian compact Lie groups satisfy the mutually equivalent conditions of \Cref{th:which.cpct.lie.gps.loc}.
  \end{enumerate}
\end{corollary}
\begin{proof}
  Toral isotropy is delivered by \Cref{pr:fin.tor.loc.ntrl}\Cref{item:pr:fin.tor.loc.ntrl:subtor.fin}, whence also virtually abelian ($=$virtually toral for compact Lie groups) by \Cref{th:fin.ind.dscnt}\Cref{item:th:fin.ind.dscnt:fin.ind.l}. 
\end{proof}

\pf{cor:virt.ab.istrp}
\begin{cor:virt.ab.istrp}
  This follows from \Cref{cor:vrt.ab} in conjunction with \Cref{th:which.cpct.lie.gps.loc}.
\end{cor:virt.ab.istrp}

Aiming at assessing the extent to which bundles$_{\bU}$ are amenable to equivariant neutralization, we remind the reader the broad plan driving the proof of \cite[Lemma 12]{zbMATH03329610} (which crucially assumes a finite-complex base).

\begin{recollection}\label{rec:k.ntrl}
  Whether or not bundles$_{\bU}$ can be neutralized equivariantly will frequently reduce by various means to actions on finite $\bU$-CW complex bases, in which generality the present outline is applicable. ``Various'' indicates some dependence on context: with \Cref{res:mtrz.as.vect}\Cref{item:res:mtrz.as.vect:plb.cpct} operative, for instance, reducing the discussion to smooth $\bU$-actions on compact connected $C^{\infty}$ manifolds $X$, recall (\cite[\S I.3, p.16]{may_eq-hotop_1996}, \cite[Corollary 7.2]{MR696520}) that such $\bU$-spaces admit finite-$\bU$-CW structures. 

  \begin{enumerate}[(I), wide]
  \item\label{item:rec:k.ntrl:k} One first proves the desired result $K$-theoretically: for every bundle class $[\cE]\in K(X)$ there is an $a\in K(X)$ with $[\cE]a\in \bZ_{>0}\subseteq K(X)$. This follows from the nilpotence \cite[Corollary 3.1.6]{at_k_1967} of $[\cE]-\mathrm{rank}~\cE$, giving a power of $\mathrm{rank}~\cE$ that is a multiple of $[\cE]$ in $K(X)$. 
    
  \item\label{item:rec:k.ntrl:stab} $X$'s being a finite complex plays no role in the preceding portion of the argument. It rather enters the discussion at this point, where it is used to ensure that the class $b=[\cW]-n$, $n\in \bZ_{>0}\subseteq K(X)$ is realizable as that of an actual vector bundle $\cW'$ of rank $\mathrm{rank}~\cW-n$ via the homotopic stability of the unitary/orthogonal groups in ranks $\gg \dim X$ (``$\gg$'' meaning ``much larger than'').
  \end{enumerate}
\end{recollection}

\begin{remark}\label{re:suff.lrg.bdl}
  The homotopy-stability phenomenon alluded to in \Cref{rec:k.ntrl}\Cref{item:rec:k.ntrl:stab} amounts to the observation that sufficiently high-rank bundles over a finite complex $X$ contain sufficiently high-rank trivial subbundles: this follows from the isomorphisms
  \begin{equation*}
    [X,U(n)]
    \xrightarrow[\quad\cong\quad]{\quad}
    [X,U(n+1)]
    ,\quad
    n\gg 0
    ,\quad
    [-,-]:=\text{homotopy map classes}
  \end{equation*}
  (and analogues for orthogonal groups for real vector bundles), consequent on \emph{Puppe sequence} \cite[Theorem 6.42(1)]{dk_at_2001} resulting from the fibrations
  \begin{equation*}
    \begin{tikzpicture}[>=stealth,auto,baseline=(current  bounding  box.center)]
      \path[anchor=base] 
      (0,0) node (l) {$U(n)$}
      +(2.5,.5) node (u) {$U(n+1)$}
      +(4,-.5) node (r) {$\bS^{2n+1}$}
      ;
      \draw[right hook->] (l) to[bend left=6] node[pos=.5,auto] {$\scriptstyle $} (u);
      \draw[->>] (u) to[bend left=6] node[pos=.5,auto] {$\scriptstyle $} (r);
    \end{tikzpicture}    
  \end{equation*}
  This can also be phrased as maps into truncated vector-bundle classifying spaces stabilizing:
  \begin{equation}\label{eq:gqv.gqwv}
    \dim V\gg q\gg 0
    \xRightarrow{\quad}
    \forall W
    \left(
      \left[X,\bG(q,V)\right]
      \ni f
      \xmapsto[\quad\cong \quad]{\quad}
      f\oplus W
      \in
      \left[X,\bG(q+\dim W,V\oplus W)\right]
    \right),
  \end{equation}
  $\bG(q,\bullet)$ denoting $q$-plane Grassmannians. 
  
  No naive $\bU$-equivariant analogue can hold for arbitrary, possibly infinite $\bU$ (e.g. by simply substituting neutral bundles for trivial, or arbitrary representations for $W$ in \Cref{eq:gqv.gqwv}). Counterexamples with infinite isotropy can easily be produced using the fact that an infinite compact group has infinitely many irreducible representations.
\end{remark}

\Cref{rec:k.ntrl}'s \Cref{item:rec:k.ntrl:k} motivates the following notion, referencing the \emph{equivariant $K$-theory} rings $K_{\bU}(X)$ attached to an action $\bU\circlearrowright X$ familiar from \cite[\S 2]{seg}, \cite[Chapter XIV]{may_eq-hotop_1996}, etc.

\begin{definition}\label{def:k.ntrl}
  An action $\bU\circlearrowright X$ of a compact Lie group on a compact Hausdorff space is \emph{$K$-theoretically} (or \emph{$K$-})\emph{neutral} if all vector bundles$_{\bU}$ $\cE\xrightarrowdbl{} X$ admit classes $a\in K_{\bU}(X)$ with $[\cE]a\in K_{\bU}(*)\subseteq K_{\bU}(X)$, the representation ring $R(\bU)\cong K_{\bU}(*)$ \cite[\S 2, Example (i)]{seg}.
\end{definition}

In a follow-up to \Cref{cor:vrt.ab}, we have the following instance of \Cref{rec:k.ntrl}\Cref{item:rec:k.ntrl:k} for sufficiently well-behaved actions. 

\begin{theorem}\label{th:fin.istrp.loc.k}
  Compact-Lie-group actions $\bU\circlearrowright X$ with finite isotropy groups $\bU_x$, $x\in X$ on compact Hausdorff spaces are both locally neutral and $K$-neutral. 
\end{theorem}
\begin{proof}
  Local neutrality follows from \Cref{cor:vrt.ab} itself via \Cref{pr:loc.ntrl}'s equivalence \Cref{item:pr:loc.ntrl:loc.ntrl} $\Leftrightarrow$ \Cref{item:pr:loc.ntrl:orb.ntrl}, so the novelty of the claim is the $K$-neutrality the proof will in fact focus on.

  Note first that an action as described, with finite isotropy on a compact space, will automatically have finitely many \emph{orbit types} (conjugacy classes of stabilizers $\bU_x$: \cite[post Corollary I.4.4]{bred_cpct-transf}); this follows from $X$ admitting a finite cover by tubes, together with isotropy groups $\bU_y$ for $y$ nearby $x$ being conjugate to subgroups of $\bU_x$ \cite[Theorem II.5.4, Corollary II.5.5]{bred_cpct-transf} in sufficiently small tubes.

  We recall next that per the proof of \cite[Theorem 9]{MR893156} (appealed to once before), \emph{any} faithful representation $\rho:\bG\le \bU(m)$ of a finite group admits a \emph{rational Schur functor}\footnote{The phrase \emph{rational} is meant here to contrast with ``regular'' or ``polynomial'', by analogy to familiar algebraic-geometric language: plain \emph{Schur functors} $\bS_{\lambda}$ are typically \cite[\S 6.1]{fh_rep-th} indexed by partitions $\lambda=(\lambda_1\ge \cdots\ge \lambda_m\ge 0)$ and provide \emph{some} of the representations of $U(m)$ when applied to the defining $m$-dimensional representation $\rho:U(m)\circlearrowright V$, but must be supplemented by tensoring with possibly negative powers of the determinant representation $\bigwedge^m V$ to recover the rest. This produces what $\cite[\S 15.5]{fh_rep-th}$ denotes by $\Psi_{\lambda}$ for $\lambda=(\lambda_1\ge \cdots\ge \lambda_n)$ (no positivity assumptions), and what the text above would render as $\bS_{\lambda}\rho$ instead.} $\bS_{\lambda}\rho$ commensurable with the regular representation $\rho_{\bG}$. Thus:
  \begin{itemize}[wide]
  \item cover $X$ with finitely many tubes about orbits $\bU x\cong \bU/\bU_x$;
  \item for each of the finitely many selected points $x\in X$ pick a representation $\rho_x\in \Rep(\bU)$ with regular-commensurable restriction $\rho_x|_{\bU_x}$;
  \item and tensor $\cE$ with all $\rho_{x,X}$.     
  \end{itemize}
  This will produce a bundle$_{\bU}$ with regular-commensurable isotropy representations
  \begin{equation*}
    \bU_y\circlearrowright \left(\cE\otimes \bigotimes_{x}\rho_{x,X}\right)_y
    ,\quad
    y\in X
  \end{equation*}
  by the regular representation's absorption property (\Cref{re:fin.loc.ntrl}), and that same property ensures that that feature will not alter by further tensoring with \emph{any} vector bundle$_{\bU}$. In summary: every $\bU_x\circlearrowright \cE_x$ can be assumed commensurable with the regular representation $\rho_{\bU_x}$. 

  We now recycle the notation and setup (closed tubes $X_x$ attached to finitely many $x$, etc.). Some $K_{\bU}$ class of the form $m[\cE]-n[\rho_{x,X}]$ vanishes at every $y\in X_x$ if every $\rho_x\in \Rep(\bU)$ has $\rho_x|_{\bU_x}\sim \rho_{\bU_x}$, so after positive-integer scaling (perhaps by distinct positive integers over finitely many members of a clopen $X$-cover) we can assume \cite[Proposition 5.1]{seg}
  \begin{equation*}
    \forall x
    \bigg([\cE]-[\rho_{x,X}]\in K_{\bU}(X_x)\text{ is nilpotent}\bigg).
  \end{equation*}
  This provides a globally nilpotent class $\prod_x \left([\cE]-[\rho_{x,X}]\right)$, and hence a neutral multiple $\left[\bigotimes_x \rho_{x,X}\right]$ of $[\cE]$ in $K_{\bU}(X)$. 
\end{proof}

Note also the following consequence of the proof of \Cref{th:fin.istrp.loc.k}, amplifying its statement.

\begin{corollary}\label{cor:k.ntrl.lrg.rep}
  Let $\bU\circlearrowright X$ be a finite-isotropy compact-Lie action on a compact $T_2$ space and $\cE\xrightarrowdbl{}X$ a vector bundle$_{\bU}$.

  There are $a\in K_{\bU}(X)$ and $\rho\in \Rep(\bU)$ with $[\cE]a=[\rho_X]$ and all  
  \begin{equation*}
    a|_x\in K_{\bU_x}(\{x\})\cong R(\bU_x)
    ,\quad
    \rho\in R(\bU)
  \end{equation*}
  arbitrarily large with respect to the usual \cite[post Corollaire 2]{MR838399} representation-ring ordering induced by the positive cone
  \begin{equation*}
    \left\{\text{characters}\right\}
    =:
    R_+(\bullet)
    \le
    R(\bullet)
    :=
    \left\{\text{virtual characters}\right\}
    ,\quad
    \left(\bullet\text{ compact Lie}\right).
  \end{equation*}
\end{corollary}
\begin{proof}
  As mentioned, this is implicit in the proof of \Cref{th:fin.istrp.loc.k} (rather than the statement): for $\rho$ there is of course no issue (for one can always tensor further by arbitrarily large representations), while the restrictions $a|_x$ can be chosen so as to dominate arbitrarily high multiples of the regular representations $\rho_{\bU_x}$ of the (finitely many, up to conjugacy) isotropy groups $\bU_x$. In fact, we can arrange for $a|_x$ to \emph{be} such multiples:
  \begin{equation}\label{eq:ax.lrg}
    a|_x\cong n_x[\rho_{\bU_x}]
    ,\quad
    n_x\ge n\gg 0:
  \end{equation}
  simply multiply both $a$ and $[\rho_X]$ by $[\eta_X]$ for suitable $\eta\in \Rep(\bU)$.
\end{proof}


\addcontentsline{toc}{section}{References}

\def\polhk#1{\setbox0=\hbox{#1}{\ooalign{\hidewidth
  \lower1.5ex\hbox{`}\hidewidth\crcr\unhbox0}}}


\Addresses

\end{document}